\documentclass[11pt]{amsart}
\usepackage{tikz-cd}
\usepackage{verbatim}
\usepackage[colorinlistoftodos]{todonotes}
\usepackage{mathtools}
\usepackage[bookmarks,colorlinks,breaklinks]{hyperref}  
\usepackage[top=2cm, bottom=4.5cm, left=3cm, right=3cm]{geometry}
\hypersetup{linkcolor=blue,citecolor=blue,filecolor=blue,urlcolor=blue} 
\newtheorem{theorem}{Theorem}[section]
\newtheorem{proposition}[theorem]{Proposition}
\newtheorem{lemma}[theorem]{Lemma}

\newtheorem{fact}[theorem]{Fact}

\newtheorem{cor}[theorem]{Corollary}
 
\newtheorem{definition}[theorem]{Definition}

\newtheorem{conjecture}[theorem]{Conjecture}
\newtheorem{notation}[theorem]{Notation}
\theoremstyle{plain}
\numberwithin{equation}{theorem}

\theoremstyle{remark}
\newtheorem{remark}[theorem]{Remark}

\newcommand{\F}{{\mathbb F}}
\newcommand{\Fpbar}{\overline{\mathbb{F}_p}}
\newcommand{\N}{{\mathbb N}}

\newcommand{\bG}{{\mathbb G}}
\newcommand{\bP}{{\mathbb P}}

\newcommand{\matO}{{\mathcal O}}
\newcommand{\Fp}{\mathbb F_p}
\newcommand{\Fq}{\mathbb F_q}

\newcommand{\lra}{\longrightarrow}
\newcommand{\dra}{\dashrightarrow}

\newcommand{\OO}{\matO}
\newcommand{\id}{{\rm id}}
\newcommand{\Pl}{{\mathbb P}}

\newcommand{\trdeg}{{\rm trdeg}}

\newcommand{\red}{\textbf{red}}

\title[Zariski dense orbits]{Zariski dense orbits for endomorphisms of a power of the additive group scheme defined over finite fields}
\author{Dragos Ghioca}  
\address{
Dragos Ghioca \\
Department of Mathematics\\
University of British Columbia\\
Vancouver, BC V6T 1Z2\\
Canada
}

\email{dghioca@math.ubc.ca}

\author{Sina Saleh}
\address{
Sina Saleh \\
Department of Mathematics\\
University of British Columbia\\
Vancouver, BC V6T 1Z2\\
Canada
}

\email{sinas@math.ubc.ca}

\keywords{Zariski dense orbits, Medvedev-Scanlon conjecture, additive polynomials over fields of positive characteristic}
\subjclass[2010]{Primary 14K15, Secondary 14G05}

\begin{document}

\begin{abstract}
We prove the Zariski dense orbit conjecture in positive characteristic for endomorphisms of $\bG_a^N$ defined over $\Fpbar$. 
\end{abstract}

\maketitle


\section{Introduction}
\subsection{Notation}
We let $\N_0:=\N\cup\{0\}$ denote the set of nonnegative integers. 

For any morphism $\Phi$ on a variety $X$ and for any integer $n\ge 0$, we let $\Phi^n$ be the $n$-th iterate of $\Phi$ (where $\Phi^0$ is the identity map $\id:=\id_X$, by definition). For a point $x\in X$, we denote by $\OO_\Phi(x)$ the orbit of $x$ under $\Phi$, i.e., the set of all $\Phi^n(x)$ for $n\ge 0$. When $\Phi$ is only a rational self-map of $X$, the orbit $\OO_\Phi(x)$ of the point $x\in X$ is well-defined if each $\Phi^n(x)$ lies outside the indeterminacy locus of $\Phi$. For any self-map $\Phi$ on a variety $X$, we say that $x\in X$ is preperiodic if its orbit $\OO_{\Phi}(x)$ is finite.

We denote by $M_{m,n}(R)$ the set of $m\times n$-matrices with entries in the ring $R$; we denote by $\mathbf{I}_m$ the identity $m\times m$-matrix.


\subsection{The classical Zariski dense orbit conjecture}

The following conjecture was motivated by a similar question raised by Zhang \cite{Zhang}, and it was formulated independently by Medvedev and Scanlon \cite{M-S} and by Amerik and Campana \cite{A-C}.

\begin{conjecture}
\label{conj:char0}
Let $X$ be a quasiprojective variety defined over an algebraically closed field $K$ of characteristic $0$ and let $\Phi:X\dra X$ be a dominant rational self-map. Then either there exists $\alpha\in X(K)$ whose orbit under $\Phi$ is well-defined and Zariski dense in $X$, or there exists a non-constant rational function $f:X\dra \Pl^1$ such that $f\circ \Phi=f$.
\end{conjecture}

There are several partial results known towards Conjecture~\ref{conj:char0} (for example, see \cite{A-C, BGSZ, BGZ, G-H, G-Sina, G-Matt, G-S, G-X, M-S}).


\subsection{The picture in positive characteristic} If $K$ has characteristic $p>0$, then Conjecture~\ref{conj:char0} does not hold due to the presence of the Frobenius endomorphism (see \cite[Example~6.2]{BGZ} and also,  \cite[Remark~1.2]{G-Sina-20}). Based on the discussion from \cite{G-Sina-20}, the authors proposed the following conjecture as a variant of Conjecture~\ref{conj:char0} in positive characteristic.

\begin{conjecture}
\label{conj:charp}
Let $K$ be an algebraically closed field of positive transcendence degree over $\Fpbar$, let $X$ be a quasiprojective variety defined over $K$, and let $\Phi:X\dra X$ be a dominant rational self-map defined over $K$ as well. Then at least one of the following three statements must hold:
\begin{itemize}
\item[(A)] There exists $\alpha\in X(K)$ whose orbit $\OO_\Phi(\alpha)$ is Zariski dense in $X$. 
\item[(B)] There exists a non-constant rational function $f:X\dra \bP^1$ such that $f\circ \Phi=f$. 
\item[(C)] There exist positive integers $m$ and $r$, there exists a variety $Y$ defined over a finite subfield $\Fq$ of $K$ such that  $\dim(Y)\ge \trdeg_{\Fpbar}K + 1$ and there exists a dominant rational map $\tau: X \dra Y$ such that
\[
\tau \circ \Phi^m = F^r \circ \tau,
\]
where $F$ is the Frobenius endomorphism of $Y$ corresponding to the field $\Fq$.
\end{itemize}
\end{conjecture}
Conjecture \ref{conj:charp} has been proven in the case of algebraic tori in \cite{G-Sina-20} and more generally in the case of all split semiabelian varieties defined over $\Fpbar$ in \cite{G-Sina-21}. For an illustration of the trichotomy in the conclusion of Conjecture \ref{conj:charp}, we refer the reader to \cite[Example 1.6]{G-Sina-20}. Also, we note that one definitely requires the hypothesis that $\trdeg_{\Fp}K\ge 1$ in Conjecture~\ref{conj:charp} since for any self-map $\Phi$ defined over $\Fpbar$, each point of $X(\Fpbar)$ is preperiodic and therefore, condition~(A) cannot hold; on the other hand, there are plenty of examples of maps $\Phi$ defined over $\Fpbar$ for which neither condition~(B) nor condition~(C) would hold.


\subsection{Our results} We prove the following more precise version of Conjecture \ref{conj:charp} in the case of group endomorphisms of $\bG_a^N$ defined over $\Fpbar$.
\begin{theorem}
\label{thm:main}
Let $N\in\N$ and let $L$ be an algebraically closed field of characteristic $p$ such that $\trdeg_{\Fpbar}L\ge 1$. Let $\Phi:\bG_a^N\lra\bG_a^N$ be a dominant group endomorphism defined over $\Fpbar$. Then at least one of the following statements  must hold.
\begin{itemize}
\item[(A)] There exists $\alpha\in \bG_a^N(L)$ whose orbit under $\Phi$ is Zariski dense in $\bG_a^N$.
\item[(B)] There exists a non-constant rational function $f:\bG_a^N\dra \bP^1$ such that $f\circ \Phi=f$.
\item[(C)] There exist positive integers $m$ and $r$, a positive integer $N_0$ greater than or equal to $\trdeg_{\Fpbar}L + 1$ and a dominant group homomorphism $\tau:\bG_a^N \lra \bG_a^{N_0}$ such that 
\begin{equation}
\label{eq:C}
\tau \circ \Phi^m = F^r\circ \tau,
\end{equation} 
where $F$ is the usual Frobenius endomorphism of $\bG_a^{N_0}$ induced by the field automorphism $x\mapsto x^p$.
\end{itemize}
\end{theorem}


\subsection{Discussion of our proof} The strategy of our proof for Theorem \ref{thm:main} is as follows. Suppose we have a group endomorphism $\Phi: \bG_a^N \lra \bG_a^N$ defined over $\Fq$ where $q = p^\ell$ for some $\ell \in \N$. As shown in \cite[Proposition~3.9]{Poguntke}, the endomorphism $\Phi$ is given by an $N$-by-$N$ matrix $A$ (acting linearly on $\bG_a^N$) whose entries are one-variable additive polynomials in the variables $x_1,\dots, x_N$ (corresponding to the $N$ coordinate axes of $\bG_a^N$), i.e.,   
\begin{equation}
\label{eq:matrix}
A=\left(f_{i,j}(x_j)\right)_{1\le i,j\le N},
\end{equation}
where a polynomial $f(x)$ is additive if it is of the form 
\begin{equation}
\label{eq:c_0}
\sum_{k=0}^r c_ix^{p^i}; 
\end{equation}
furthermore, since $\Phi$ is defined over $\Fq$, then each coefficient of each additive polynomial $f_i(x_j)$ belongs to $\Fq$. We denote by $F$ the Frobenius endomorphism (of $L$) corresponding to the finite field $\Fp$ (i.e., $x\mapsto x^p$). So, the action of $\Phi$ is given by a matrix of polynomials (with coefficients in $\Fq$) in the Frobenius operator, i.e., the entries of our matrix $A$ live in $\Fq[F]$. We will study technical properties of $\Fq[F]$ and of the ring of matrices with entries in $\Fq[F]$ in Section~\ref{sec:fieldK}.

Now, for any given point $\alpha \in \bG_a^N(L)$, the orbit $\OO_{\Phi}(\alpha)$ is contained in a finitely generated $\Fp[F]$-submodule $\Gamma$ of $\bG_a^N$. If we assume that $\OO_{\Phi}(\alpha)$ is not Zariski dense, then it must be contained in some proper subvariety $V$ of $\bG_a^N$. Then, we can describe the structure of $V(L) \cap \Gamma$ using \cite[Theorem 2.6]{G-TAMS}; for more details, see Section~\ref{subsec:M-L}. Then we will use the fact that $\Phi$ is integral over the commutative ring $\Fp[F^\ell]$; for more details, see  Section~\ref{sec:fieldK}. We also employ several reductions discussed in Section \ref{sec:reductions} which allow us to split the action of $\Phi$ into $(\Phi_1, \Phi_2)$ where $\Phi_1$ is given by the diagonal action of powers of the Frobenius endomorphism, while the minimal polynomial of $\Phi_2$ over $\Fp[F^\ell]$ has roots that are multiplicatively independent with respect to $F$. This helps us reduce Theorem~\ref{thm:main} into two separate extreme cases which are much more convenient to deal with; the general case in Theorem~\ref{thm:main} then follows.


\subsection{Discussion of possible extensions}

We note that our approach does not generalize to the case the endomorphism $\Phi$ of $\bG_a^N$ is defined over an arbitrary field $L$ of characteristic $p$. The reason is that in the case of general group endomorphisms of $\bG_a^N$, the orbit will not necessarily be contained in a finitely generated $\Fp[F]$-submodule of $\bG_a^N$, and thus, we cannot use the $F$-structure theorem proven in \cite[Theorem 2.6]{G-TAMS} anymore. In the special case when for each $1\le i,j\le N$,  the linear term $c_0$ of $f_i(x_j)$ (i.e., $c_0=f_i'(x_j)$ from equations~\eqref{eq:matrix}~and~\eqref{eq:c_0}) is transcendental over $\Fp$, then one can reformulate Conjecture~\ref{conj:charp} for $(\bG_a^N,\Phi)$ in the context of Drinfeld modules of generic characteristic. Even though there is a very rich arithmetic theory for Drinfeld modules of generic characteristic built in parallel to the classical Diophantine geometry questions in characteristic $0$ (see, for example, \cite{Breuer, Siegel-2, Siegel, GT-Compositio, DMM}), there are still several technical difficulties to overcome in this case alone. Furthermore, when we deal with the most general case of an endomorphism of $\bG_a^N$ (defined over a field $L$ of characteristic $p$), in which case some of the derivatives of the polynomials $f_i(x_j)$ are in a finite field, while others are transcendental over $\Fp$, then there are significant more complications since then we would be dealing with a mixed arithmetic structure coming from the action of Drinfeld modules of both generic characteristic and also of special characteristic (see \cite{G-IMRN} for a sample of difficulties arising in the context of Drinfeld modules of special characteristic).


\subsection{Plan for our paper} In Section \ref{sec:reductions} we use the technical results proven in Sections~\ref{sec:background}~and~\ref{sec:fieldK} to show that instead of proving Theorem~\ref{thm:main} for the dynamical system $(\bG_a^N, \Phi)$, we can instead prove it for the dynamical system $(\bG_a^{N_1} \times \bG_a^{N_2}, (\Phi_1, \Phi_2))$ where $N = N_1 + N_2$, $\Phi_1$ is a group endomorphism of $\bG_a^{N_1}$ given by the coordinate-wise action of powers of the Frobenius endomorphism, and the minimal polynomial of $\Phi_2$ over $\Fp[F^\ell]$ has roots that are multiplicatively independent with respect to $F$. In other words, we reduce Theorem~\ref{thm:main} to Proposition~\ref{prop:split2}. We conclude our paper by proving Proposition~\ref{prop:split2} in Section~\ref{sec:mainThm}.  


\section{Technical Background} \label{sec:background}

In this Section we gather some useful results for our proofs which come from two different sources: the theory of skew fields (see Section~\ref{subsec:background} and more generally \cite{Cohn}) and the $F$-structure theorem proven by Moosa-Scanlon \cite{Moosa-S} in order to describe intersections of subvarieties of $\bG_m^N$ with finitely generated subgroups of $\bG_m^N\left(\overline{\Fp(t)}\right)$ (see Lemma~\ref{lem:lambda^m} in Section~\ref{subsec:Diophantine}).


\subsection{Some results about splitting matrices over skew fields} 
\label{subsec:background}

In this Section we state useful results about splitting matrices over skew fields which will be used later in our proofs; for more details about skew fields, we refer the reader to \cite{Cohn}. 
\begin{fact}
\label{fact:background}
Let $K$ be a skew field with centre $k$. Suppose that $A$ is a matrix in $M_{n,n}(K)$ which is algebraic over $k$. Let $f(x) = f_0(x)f_1(x)$ be the minimal polynomial of $A$ over $k$ where $f_0, f_1 \in k[x]$ are coprime. Then, $A$ has a conjugate of the form
\[
A_0 \oplus A_1
\]
such that the minimal polynomials of $A_0$ and $A_1$ over $k$ are equal to $f_0$ and $f_1$, respectively.
\end{fact}
\begin{proof}
Using the arguments after \cite[Corollary~8.3.4,~p.~381--382]{Cohn}, $A$ must have a conjugate of the form 
$$
B_1 \oplus \cdots \oplus B_r,
$$
where each $B_i$ corresponds to an elementary divisor, say $q_i$. Since $f_0$ and $f_1$ are coprime, each $q_i$ must divide exactly one of $f_0$ and $f_1$ and be coprime with respect to the other one. So, assume without loss of generality that for some $i \ge 0$ we have 
\[
q_j \mid f_0 \text{ and } (q_j, f_1) = 1,
\]
for all $j \le i$ and 
\[
q_j \mid f_1 \text{ and } (q_j, f_0) = 1,
\]
for $j > i$. Letting
$$
A_0 = B_1 \oplus \cdots \oplus B_i, \text{ and } A_1 := B_{i + 1} \oplus \cdots \oplus B_r,
$$
gives us the desired conclusion.
\end{proof}
\begin{fact}
\label{fact:jordan-normal-form}
Let $K$ be a skew field with centre $k$ and $A \in M_{n,n}(K)$ be a matrix with a minimal polynomial equal to $(x - \alpha)^r$ for some $\alpha \in k$ and $r \in \N$. Then, there exist an invertible matrix $P \in M_{n,n}(K)$ such that 
\[
P^{-1}AP = J_{\alpha, r_1} \bigoplus \cdots \bigoplus J_{\alpha, r_m},
\]
where $J_{\alpha,s}$ is the $s$-by-$s$ Jordan canonical matrix having unique eigenvalue $\alpha$ and its only nonzero entries away from the diagonal being the entries in positions $(i,i+1)$ (for $i=1,\dots, s-1$), which are all equal to $1$. 
\end{fact}
\begin{proof}
This is a consequence of the the discussion after \cite[Corollary~8.3.4,~p.~381--384]{Cohn}; also, see \cite[Theorem~8.3.6]{Cohn}.
\end{proof}

\subsection{A special type of Diophantine equation} 
\label{subsec:Diophantine}

In this Section we prove Lemma~\ref{lem:lambda^m} that gives an asymptotic upper bound on the number of solutions of a special type of equation given by equation~\eqref{eqn:lambda^m}. This bound will be instrumental in our proof Theorem \ref{thm:main}. We start with an easy result which will be used in our proof for Lemma~\ref{lem:lambda^m}. The result is actually a consequence of Vandermonde determinants and still holds if we know that the equation \eqref{eqn:sumOfPowers} holds for $r$ consecutive $n$'s. 

\begin{lemma}
\label{lem:sumsOfPowers}
Let $K$ be a field. Suppose that distinct non-zero elements $\lambda_1, \dots, \lambda_r \in K$ are given. Moreover, suppose that for some $N > 0$ and $c_1, \dots, c_r \in K$ we have
\begin{equation}
\label{eqn:sumOfPowers}
c_1\lambda_1^n + \cdots + c_r\lambda_r^n = 0,
\end{equation}
for every $n \ge N$. Then, $c_1 = \cdots = c_r = 0$.
\end{lemma}
Before proving the main result of this Section, we recall that for a subset $S\subseteq \N$, the (upper asymptotic) density (also called \emph{natural density}) of $S$ is defined as
$$\mu(S):=\limsup_{n\to\infty} \frac{\#\{m\in S\colon m\le n\}}{n}.$$
\begin{lemma}
\label{lem:lambda^m}
Let $K = \overline{\Fp(t)}$, let $r\in\N$ and let $c_0,c_1,\dots, c_r\in K$.  Suppose that $\lambda \in K\setminus\{0\}$ is multiplicatively independent with respect to $t$. Let $S$ be the set of positive integers $m$ for which there exist positive integers $n_1, \dots, n_r$ such that 
\begin{equation}
\label{eqn:lambda^m}
\lambda^m = c_0 + \sum_{i=1}^r c_i t^{n_i}.
\end{equation}
Then the natural density of $S$ is equal to zero. 
\end{lemma}
\begin{proof}
Solving equation \eqref{eqn:lambda^m} is equivalent with analyzing the intersection of the hyperplane $V\subset \bG_m^{1+r}$ given by the equation:
\[
y =c_0+ \sum_{i=1}^r c_i x_i
\]
with the subgroup $\Gamma$ of $\bG_m^{1+r}$ spanned by $
(\lambda, 1, \dots, 1)$,
$(1, t, 1, \dots , 1)$,
$\dots$
$,(1, \dots , 1, t)$. Then by Moosa-Scanlon's $F$-structure theorem (see \cite[Theorem~B]{Moosa-S}), we know the intersection is a union of finitely many
sets $R_1,\dots, R_u$; furthermore, each set $R$ from the list $R_1,\dots, R_u$ is of the form
\[
R:=\gamma_0 \cdot S(\gamma_1, \dots, \gamma_s; \delta_1, \dots, \delta_s) \cdot H,
\]
where $H$ is a subgroup of $\Gamma$, while $\gamma_i\in \bG_m^{1+r}(K)$ for each $i=0,\dots, s$ and  $\delta_j\in \N$ for each $j=1,\dots, s$, and the set $S(\gamma_1,\dots, \gamma_s;\delta_1,\dots ,\delta_s)$ is defined as follows:
$$S(\gamma_1,\dots, \gamma_s;\delta_1,\dots ,\delta_s):=\left\{\prod_{j=1}^s \gamma_i^{p^{k_i\delta_i}}\colon k_i\in\N\text{ for }i=1,\dots, s \right\}.$$
So, the set $R$ consists of points of the form
\begin{equation}
\label{eqn:F-orbit-points}    
\gamma_0 \cdot  \prod_{j=1}^s \gamma_i^{p^{k_i\delta_i}} \cdot \epsilon, 
\end{equation}
where $\epsilon$ is in the subgroup $H$; moreover, there exists a positive integer $\ell$ such that $\gamma_i^\ell \in \Gamma$ for each $i=0,\dots, m$. For more details regarding the $F$-sets structure for the intersection of a subvariety of a semiabelian variety with a finitely generated subgroup, we refer the reader both to \cite{Moosa-S} and also to \cite{G-TAMS}, for further refinements of Moosa-Scanlon's original result.

Now, in order to show that the set of all positive integers $m$ for which equation \eqref{eqn:lambda^m} is
solvable has natural density equal to $0$ (as a subset of $\N$), it suffices to prove that the
projection of $H$ on the first coordinate of $\bG_m^{1+r}$ is trivial; this way,
the set of those $m$'s that satisfy equation \eqref{eqn:lambda^m} would be  a sum
of powers of $p$ (due to equation \eqref{eqn:F-orbit-points}) and thus, it would have natural density zero.

Let us assume, for the sake of contradiction, that $H$ projects non-trivially on the first coordinate of
$\bG_m^{1+r}$. This means that there exists some tuple
\[
(m_0, \ell_1,\dots, \ell_r)
\]
with $m_0\ne 0$ such that a coset of the cyclic subgroup $H_0 $ spanned by
\[
(\lambda^{m_0}, t^{\ell_1},\dots,t^{\ell_r})
\]
would be contained in the intersection of $V$ with the
subgroup $\Gamma$. 

So, letting $\ell_0:=0$, there must exist some constants $d_i\in\overline{\Fp(t)}$ (depending on 
the $c_i$'s) such that for all positive integers $n$, we have
\begin{equation}
\label{eqn:lambda^m-2}
\lambda^{m_0 n} = \sum_{i=0}^r d_i t^{\ell_i n}.    
\end{equation}
Combining the terms with the same exponent, we may assume without loss of generality that the powers $\ell_1, \dots, \ell_r$ are distinct. Using equation~\eqref{eqn:lambda^m-2} and Lemma~\ref{lem:sumsOfPowers} we get that $\lambda^{m_0}$ must be equal to $t^{\ell_i}$ for some $i=0,1,\dots, r$ (note that not all of the $d_i$'s could equal $0$ since $\lambda\ne 0$), which means that $\lambda$ is  multiplicative dependent with respect to $t$. This contradicts our hypothesis and thus delivers the conclusion of Lemma~\ref{lem:lambda^m}. 
\end{proof}


\section{Arithmetic and algebraic properties for rings involving the Frobenius operator}
\label{sec:fieldK}

From now on in this paper, we let $p$ be a prime number and let $F$ be the Frobenius operator corresponding to the field $\Fp$.


\subsection{Operators involving the Frobenius operator}
\label{subsec:F}

We consider the polynomial ring $\Fp[F]$ whose elements are operators of the form $\sum_{i=0}^n a_iF^i$ which act on any field $L$ of characteristic $p$ as follows:
\begin{equation}
\label{eq:acting}
\left(\sum_{i=0}^n a_iF^i\right)(x)=\sum_{i=0}^n a_i x^{p^i}\text{ for }x\in L.
\end{equation}
Since $F$ leaves invariant each element of $\Fp$, we can identify $\Fp[F]$ with a polynomial ring in one variable over $\Fp$; in particular, we can consider its fraction field, denoted $\Fp(F)$. For an algebraically closed field $L$ (of characteristic $p$), the action of an element $u$ of $\Fp(F)$, which is of the form 
\begin{equation}
\label{eq:denominator}
u:=\frac{\sum_{i=0}^na_iF^i}{\sum_{j=0}^m b_jF^j}
\end{equation}
(for some non-negative integers $m$ and $n$ and moreover, the denominator in equation~\eqref{eq:denominator} is nonzero and is coprime with respect to the numerator) can be interpreted as a finite-to-finite map $\varphi_u:L\lra L$ which has the property that to each element $x\in L$, it associates the finitely many elements $y\in L$ for which
\begin{equation}
\label{eq:finite-to-finite}
\left(\sum_{j=0}^m b_jF^j\right)(y)=\left(\sum_{i=0}^n a_iF^i\right)(x).
\end{equation}


\subsection{Linearly independent elements with respect to the maps from the  polynomial ring in the Frobenius operator}

In this Section, we let $L$ be an algebraically closed field of positive transcendence degree over $\Fp$. The following notion is used in our proof of Proposition~\ref{prop:split2}, which is a key technical step in deriving our main result (Theorem~\ref{thm:main}). First, we note that similar to our construction of the ring of operators $\Fp[F]$ from Section~\ref{subsec:F}, we can construct the non-commutative ring of operators $\Fpbar[F]$. 

\begin{definition}
\label{def:independent}
Given elements $\delta_1,\dots, \delta_\ell$ and $\gamma_1,\dots, \gamma_k$ in $L$, we say that $\gamma_1,\dots, \gamma_k$ are linearly independent from $\delta_1,\dots, \delta_\ell$ over $\Fpbar[F]$ if whenever there exist polynomial operators $P_1(F),\dots,P_k(F)\in \Fpbar[F]$ and $Q_1(F),\dots, Q_\ell(F)\in\Fpbar[F]$ such that
$$\sum_{i=1}^k P_i(F)(\gamma_i)=\sum_{j=1}^\ell Q_j(F)(\delta_j),$$
then we must have that $P_1(F)=\cdots =P_k(F)=0$.

Moreover, in the special case $\ell=1$ and $\delta_1=\{0\}$, then the above condition simply translates into asking that the points $\gamma_1,\dots, \gamma_k$ are linearly independent over $\Fpbar[F]$.
\end{definition}

The following result will be used in our proof of Proposition~\ref{prop:split2}.
\begin{proposition}
\label{prop:independent}
For any positive integers $k$ and $\ell$, and any given elements $\delta_1,\dots, \delta_\ell\in L$, there exist $\gamma_1,\dots, \gamma_k\in L$ which are linearly independent from $\delta_1,\dots, \delta_\ell$ over $\Fpbar[F]$.
\end{proposition}

\begin{proof}
We let $L_0\subset L$ be a finitely generated extension of $\Fpbar$ containing $\delta_1,\dots, \delta_\ell$. Then we view $L_0$ as the function field of a projective, smooth variety $V$ defined over $\Fpbar$. We let $\Omega_V$ be the set of inequivalent absolute values corresponding to the irreducible divisors of $V$. Since there are only finitely many places of $V$ where the $\delta_j$'s have poles, we can choose $k$ elements $\gamma_i\in L_0$ (for $i=1,\dots, k$) such that there exist absolute values $|\cdot|_{v_i}\in \Omega_V$ (for $i=1,\dots, k$) satisfying the following properties:
\begin{itemize}
\item[(i)] $|\delta_j|_{v_i}\le 1$ for each $1\le j\le \ell$ and each $1\le i\le k$;
\item[(ii)] $|\gamma_i|_{v_i}>1$ for each $i=1,\dots ,k$; and
\item[(iii)] $|\gamma_i|_{v_j}\le 1$ for each $j\ne i$.
\end{itemize}
Conditions~(i)-(iii) can be achieved since there exist infinitely many absolute values in $\Omega_V$ and so, we can proceed inductively on $k$, each time choosing an element $\gamma_i$ which has a pole at some place of $V$ where none of the $\delta_j$'s and also none of the $\gamma_1,\dots, \gamma_{i-1}$ have poles.

Now, we claim that the elements $\gamma_1,\dots,\gamma_k$ are linearly independent from $\delta_1,\dots,\delta_\ell$ over $\Fpbar[F]$. Indeed, if there exist polynomial operators $P_1(F),\dots, P_k(F)\in\Fpbar[F]$ and $Q_1(F),\dots, Q_\ell(F)\in\Fpbar[F]$ such that
\begin{equation}
\label{eq:independent_2}
\sum_{i=1}^k P_i(F)(\gamma_i)=\sum_{j=1}^\ell Q_j(F)(\delta_j),
\end{equation}
then we assume there exists some $i_0\in\{1,\dots, k\}$ such that $P_{i_0}(F)\ne 0$ and we will derive a contradiction. Indeed, using conditions~(ii)-(iii) above,  we get that
\begin{equation}
\label{eq:pole_1}
\left|\sum_{i=1}^k P_i(F)(\gamma_i)\right|_{v_{i_0}} = \left|P_{i_0}(F)\left(\gamma_{i_0}\right)\right|_{v_{i_0}}>1.
\end{equation}
Note that in order to derive inequality~\eqref{eq:pole_1}, we use the fact that if $|\gamma|_v>1$, then for any nonzero polynomial operator $P(F)\in\Fpbar[F]$ of degree $D\ge 0$ (in the operator $F$), we have that 
$$\left|P(F)(\gamma)\right|_v=|\gamma|_v^{p^D}.$$
On the other hand, using condition~(i) above, we get that
\begin{equation}
\label{eq:pole_2}
\left|\sum_{j=1}^\ell Q_j(F)(\delta_j)\right|_{v_{i_0}}\le 1.
\end{equation}
Inequalities \eqref{eq:pole_1} and \eqref{eq:pole_2} yield a contradiction along with equality~\eqref{eq:independent_2}. This shows that indeed, the elements $\gamma_1,\dots, \gamma_k$ must be linearly independent from $\delta_1,\dots, \delta_\ell$ over $\Fpbar[F]$, which concludes our proof for Proposition~\ref{prop:independent}.
\end{proof}


\subsection{A Mordell-Lang type theorem for the additive group scheme}
\label{subsec:M-L}

We will employ in our proof of Proposition~\ref{prop:split2}  a Mordell-Lang type theorem for the additive group scheme, which was proven in \cite[Theorem~2.6]{G-TAMS}. Before stating our technical result (see Proposition~\ref{prop:M-L}), we need to introduce some notation. 

Let $N$ be a positive integer and we extend the action of the ring of operators $\Fp[F]$ on $\bG_a^N$ acting diagonally. Let $L$ be an algebraically closed field of characteristic $p$. Inspired by the definition of $F$-sets introduced by Moosa and Scanlon in \cite{Moosa-S}, we define the following subsets of $\bG_a^N(L)$. So, for points $\gamma_1,\dots, \gamma_r\in \bG_a^N(L)$ and positive integers $k_1,\dots, k_r$, we define 
\begin{equation}
\label{eq:S-additive}
S(\gamma_1,\dots, \gamma_r;k_1,\dots, k_r):=\left\{\sum_{i=1}^r F^{n_ik_i}(\gamma_i)\colon n_i\in\N\text{ for }i=1,\dots, r\right\}.
\end{equation}
The following result is proven in \cite[Theorem~2.6]{G-TAMS}.
\begin{proposition}
\label{prop:M-L}
Let $X\subseteq \bG_a^N$ be an affine variety defined over an algebraically closed field $L$ of characteristic $p$. Let $F$ be the usual Frobenius map $x\mapsto x^p$ and we extend the action of $\Fp[F]$ to $\bG_a^N$ acting diagonally. Let $\Gamma\subset \bG_a^N(L)$ be a finitely generated $\Fp[F]$-submodule. Then the intersection $X(L)\cap\Gamma$ is a finite union of sets of the form
\begin{equation}
\label{eq:description_M-L}
\gamma_0+S(\gamma_1,\dots,\gamma_r;k_1,\dots, k_r)+H,
\end{equation}
for some points $\gamma_0,\gamma_1,\dots, \gamma_r\in \bG_a^N(L)$ and positive integers $k_1,\dots, k_r$, where $S(\gamma_1,\dots,\gamma_r;k_1,\dots, k_r)$ is defined as in equation~\eqref{eq:S-additive}, while $H$ is an $\Fp[F]$-submodule of $\Gamma$. 
\end{proposition}
Proposition~\ref{prop:M-L} can be viewed as a Mordell-Lang type statement for the additive group scheme, in the same spirit as Moosa-Scanlon's main result from \cite{Moosa-S} (which is a Mordell-Lang theorem for semiabelian varieties defined over finite fields). Actually, the proof of \cite[Theorem~2.6]{G-TAMS} followed the exact strategy employed by Moosa-Scanlon for proving \cite[Theorem~7.8]{Moosa-S}. Both \cite[Theorem~2.6]{G-TAMS} and \cite[Theorem~B]{Moosa-S} (and also their common generalization for arbitrary commutative algebraic groups proven in \cite{G-new}) are extensions in positive characteristic of the classical Mordell-Lang Theorem proven by Faltings \cite{Faltings} (for abelian varieties) and by Vojta \cite{Vojta} (for semiabelian varieties).

\begin{remark}
\label{rem:important_M-L}
We make a couple of important observations regarding Proposition~\ref{prop:M-L}. First, the statement in \cite[Theorem~2.6]{G-TAMS} assumed the points in $\Gamma$ live in a finitely generated, regular extension of some given finite field; this can always be achieved and also it is not essential for the proof, as observed in \cite[Remark~7.11]{Moosa-S} (and also noted before the statement of \cite[Theorem~2.2]{G-new}).

Second, just as shown in \cite[Lemma~2.7]{Moosa-S}, one can prove that the points $\gamma_0,\gamma_1,\dots, \gamma_r$ corresponding to an $F$-set as appearing in the intersection of $X(L)\cap\Gamma$ from equation~\eqref{eq:description_M-L} live in the $\Fp[F]$-division hull of $\Gamma$, i.e., there exists some nonzero polynomial $P(F)\in \Fp[F]$ such that 
\begin{equation}
\label{eq:in_Gamma}
P(F)(\gamma_i)\in \Gamma\text{ for }i=0,1,\dots, r.
\end{equation}  
\end{remark}


\subsection{A non-commutative ring of operators}
\label{subsec:K}

From now on, we fix $q = p^\ell$ for some given $\ell \in \N$. Then the polynomial ring of operators $\Fq[F]$ acting as in equation~\eqref{eq:acting} is no longer a commutative ring since for some $c\in \Fq\setminus \Fp$, we have that $c^p\ne c$. 

From now on, we define $K := \Fq[F] \otimes_{\Fp[F^\ell]}\Fp(F^\ell)$ and so, any element in $K$ can be written as 
\begin{equation}
\label{eq:K}
\sum_{i=0}^{\ell -1}a_iF^i,
\end{equation}
where $a_i \in \Fq[F^\ell] \otimes_{\Fp[F^\ell] }\Fp(F^\ell) \subseteq \Fq(F^\ell)$; note that $\Fq(F^\ell)$ is a field since $\F^\ell$ fixes each element of $\Fq$. 


\subsection{Matrices of operators}
\label{subsec:matrices}
 
In this Section, we study matrices whose entries are themselves operators from $K$ (see the notation from Section~\ref{subsec:K}). 

\begin{notation}
\label{not:not}
Let $A$ be a matrix in $M_{n,n}(K)$ for some $n \in \N$. Using equation~\eqref{eq:K}, we can find unique matrices $A_0, \dots, A_{\ell - 1} \in M_{n,n}(\Fq(F^{\ell}))$ such that 
\[
A = \sum_{i=0}^{\ell - 1}A_iF^{i}. 
\]
From now on, we will use the matrices $A_0, \dots, A_{\ell - 1}$ each one of them belonging to $M_{n,n}(\Fq(F^\ell))$ in order to identify the matrix $A\in M_{n,n}(K)$. Furthermore, for convenience, we will often use the $n\times n\ell$-matrix $(A_0,\dots, A_{\ell-1})\in M_{n,n\ell}(\Fq(F^\ell))$ to identify the matrix $A\in M_{n,n}(K)$.
\end{notation}

The next result is a simple consequence of multiplying matrices from $M_{n,n}(K)$ and keeping track of the decomposition of their action as given in Notation~\ref{not:not}. 
\begin{proposition}
\label{prop:Atilde}
Given a matrix $A \in M_{n, n}(K)$ there exists a unique matrix $\tilde{A} \in M_{n\ell, n\ell}(\Fq(F^{\ell}))$ such that for every matrix $B \in M_{n, n}(K)$ we have 
\begin{equation}
\label{eq:identity}
\left((BA)_0, \dots, (BA)_{\ell - 1}\right) = \left(B_0, \dots, B_{\ell - 1}\right)\cdot \tilde{A}.
\end{equation}
\end{proposition}
\begin{proof}
The uniqueness of $\tilde{A}$ is obvious due to equation~\eqref{eq:identity} since we can take each matrix $B_i$ to be any arbitrary matrix in $M_{n,n}(\Fq(F^\ell))$. 

Now, we identify as in Notation~\ref{not:not} the matrix $A\in M_{n,n}(K)$ with the vector of matrices $(A_0,\dots, A_{\ell-1})$, each one of them in $M_{n,n}(\Fq(F^\ell))$. 
We define the function ${\red}:\{0, \dots, \ell - 1\} \times \{0, \dots, \ell - 1\} \lra \{0, \dots, \ell - 1\}$ to be the map given by 
\[
(i, j) \mapsto (i - j) \pmod{\ell}
\]
We define then 
\[
\tilde{A} = \left(F^jA_{{\red}(i, j)}F^{{\red}(i, j) - i}\right)_{0\le i,j\le \ell -1},
\]
which we view as an $\ell\times \ell$-matrix whose entries are themselves matrices from $M_{n, n}(\Fq[F])$. The fact that the entries of $\tilde{A}$ lie inside $\Fq(F^{\ell})$ is clear from the fact that
\[
F^jA_{{\red}(i, j)}F^{{\rm red}(i, j)-i} = \left(F^jA_{{\red}(i, j)}F^{-j}\right)F^{j + {\red}(i, j) - i}.
\]
Indeed, the entries of $F^jA_{{\red}(i, j)}F^{-j}$ all lie inside $\Fq(F^{\ell})$ since the entries of $A_{\red(i, j)}$ lie inside $\Fq(F^{\ell})$; furthermore, $F^{j + {\red}(i,j) - i}$ is a non-negative power of $F^{\ell}$.  
\end{proof}

The following result is an immediate consequence of the definition of $\tilde{A}\in M_{n\ell,n\ell}(\Fq(F^\ell))$ for any given matrix $A\in M_{n,n}(K)$ satisfying the conclusion of Proposition~\ref{prop:Atilde}. 
\begin{proposition}
\label{prop:embedding}
The map $M_{n, n}(K) \lra M_{n\ell, n\ell}(\Fq(F^\ell))$ given by 
\[
A \mapsto \tilde{A}
\]
is an embedding of $\Fp[F^\ell]$-algebras.
\end{proposition}


\subsection{A skew field}
\label{subsec:H-C}

Finally, in this Section we prove that $K$ is a skew field with center $\Fp(F^\ell)$.

The next result is an easy consequence of Proposition~\ref{prop:embedding}.
\begin{cor}
For every $A \in M_{n, n}(\Fq[F])$ there exists a monic polynomial $Q(x) \in \Fp[F^\ell][x]$ such that $Q(A) = 0$.
\end{cor}
\begin{proof}
Since the $n\ell\times n\ell$-matrix $\tilde{A}$ has its entries in the commutative ring $\Fq[F^\ell]$, then the classical Cayley-Hamilton's theorem yields the existence of a monic polynomial with coefficients in $\Fq[F^\ell]$ which kills the matrix $\tilde{A}$. Because $\Fq[F^\ell]$ is itself integral over $\Fp[F^\ell]$, then we can find a monic polynomial $Q(x) \in \Fp[F^\ell][x]$ such that $Q(\tilde{A}) = 0$. Then, Proposition \ref{prop:embedding} yields that $Q(A) = 0$ as well. 
\end{proof}
\begin{proposition}
\label{prop:QP}
 For any $P(F) \in \Fq[F]$ there exists a nonzero polynomial $Q(F) \in \Fq[F]$ such that $Q(F)P(F) \in \Fp[F^\ell]$. 
\end{proposition}
\begin{proof}
We regard $P(F)$ as a matrix in $M_{1,1}(\Fq[F])$. Since $P$ is non-zero, then Proposition~\ref{prop:embedding} yields that  $\tilde{P}$ must be an invertible $\ell\times \ell$-matrix. Therefore, there exists a vector $(Q_0, \dots, Q_{\ell - 1})$ with coordinates in $\Fq[F^\ell]$ such that 
\begin{equation}
\label{eqn:QP}
(Q_0, \dots, Q_{\ell - 1}) \cdot \tilde{P} = (\alpha, 0, \dots, 0).
\end{equation}
for some non-zero $\alpha \in \Fq[F^\ell]$. If we let $Q$ be the (nonzero)  polynomial in $\Fq[F^\ell]$ corresponding to $(Q_0, \dots, Q_{\ell - 1})$, then equation~\eqref{eqn:QP} implies that $Q_1(F^\ell) := Q(F)P(F) \in \Fq[F^\ell]$. Since $\Fq$ is a finite extension of $\Fp$, there must exist another nonzero  polynomial $Q_2 \in \Fq[F^\ell]$ such that $Q_2(F^\ell)Q_1(F^\ell) \in \Fp[F^\ell]$. So, 
\[
Q_2(F^\ell)Q(F)P(F) \in \Fp[F^\ell]
\]
as desired. 
\end{proof}

Finally, the desired conclusion about $K$ being a skew field with center $\Fp(F^\ell)$ follows as an immediate consequence of Proposition~\ref{prop:QP}.
\begin{cor}
\label{cor:skewField}
$\Fq[F] \otimes_{\Fp[F^\ell]} \Fp(F^\ell)$ is a skew field and $\Fp(F^\ell) $ is its centre.
\end{cor}


\section{Reductions for our main result}\label{sec:reductions}

\begin{proposition}
\label{prop:iterate-reduction}
In order to prove Theorem \ref{thm:main} for the dynamical system $(\bG_a^N,\Phi)$, it suffices to prove Theorem~\ref{thm:main} for the dynamical system $(\bG_a^N, \Phi^n)$ for some $n\in\N$. 
\end{proposition}
\begin{proof}
It is clear that if condition~(C) holds for an iterate of $\Phi$ then it also holds for $\Phi$. The fact that if conditions~(A)~and~(B) hold for an iterate of $\Phi$ then they also hold for $\Phi^n$ follows from \cite[Lemma~2.1]{BGSZ}. 
\end{proof}

\begin{notation}
Let $h$ be an element in $\Fpbar[F]$ and let $N$ be a positive integer. We let $[h]$ denote the group endomorphism of  $\bG_a^N$ given by the coordinate-wise action of $h$. 
\end{notation}

Also, as a matter of notation thoughout our paper, we will often use $\vec{x}$ to denote the point $x\in\bG_a^N$ just so it would be more convenient when using a group endomorphism $\Phi$ of $\bG_a^N$ corresponding to some matrix $A\in M_{N,N}(\Fpbar[F])$, because then we would write $A\vec{x}$ to denote $\Psi(\vec{x})$.

\begin{definition}
\label{def:finite-to-finite}
We call $\Psi:\bG_a^N\lra \bG_a^N$ a finite-to-finite map (defined over $\Fpbar$) if there exists a nonzero element $h\in\Fpbar[F]$ with the property that $[h]\circ \Psi$ is a group endomorphism of $\bG_a^N$. In other words, there exists a matrix $B\in M_{N,N}(\Fpbar[F])$ such that for each point $x\in \bG_a^N$, the finite-to-finite  map $\Psi$ associates to the point $x$ the finitely many points $y\in\bG_a^N$ such that
$$[h](\vec{y})=B\vec{x}.$$
\end{definition}

\begin{proposition}
\label{prop:split}
Let $K = \Fq[F] \otimes_{\Fp[F^\ell]} \Fp(F^\ell)$ where $q = p^\ell$ for some $\ell \in \N$ (see also Section~\ref{subsec:K}). Let $\Phi: \bG_a^N \lra \bG_a^N$ be a dominant group endomorphism of $\bG_a^N$ defined over $\Fq$. Then, there exists $n\in\N$ and there exist  non-negative integers $N_0$ and $N_1$ such that $N = N_0 + N_1$, along with a dominant group endomorphism $\Phi_0: \bG_a^{N_0} \lra \bG_a^{N_0}$  corresponding to the matrix
\begin{equation}
\label{eqn:F-jordan}
A_0 := J_{F^{n_1\ell}, m_1} \bigoplus \cdots \bigoplus J_{F^{n_s\ell}, m_s},
\end{equation}
and a finite-to-finite map $\Phi_1: \bG_a^{N_1} \lra \bG_a^{N_1}$ (see also Definition~\ref{def:finite-to-finite}) corresponding to a matrix $A_1 \in M_{N_1, N_1}(K)$,  where the minimal polynomial of $A_1$ over $\Fp[F^\ell]$ has roots that are multiplicatively independent with respect to $F^\ell$, there exists a dominant group endomorphism $g: \bG_a^N \lra \bG_a^{N_0} \times \bG_a^{N_1}$ defined over $\Fq$, and there exists a nonzero  element $h \in \Fp[F^\ell]$ such that the next diagram commutes
\begin{equation}
\label{eqn:splitPhi}
\begin{tikzcd}
\bG_a^N \arrow[r, "\text{$[h]\circ\Phi^{mn}$}"]   \arrow[d, "g"] &[2em] \bG_a^N \arrow[d, "g"] \\
\bG_a^{N_0} \times \bG_a^{N_1} \arrow[r, "\text{$[h]\circ (\Phi_0^m, \Phi_1^m)$}"] & \bG_a^{N_0} \times \bG_a^{N_1},
\end{tikzcd}    
\end{equation}
for all $m \in \N$. In particular, $[h]\circ (\Phi_0^m, \Phi_1^m)$ is a well-defined group endomorphism for all $m \in \N$.
\end{proposition}
\begin{proof}
Suppose that $\Phi$ corresponds to a matrix $A \in M_{N, N}(\Fq[F])$. For some suitable power $\Phi^n$ of $\Phi$ we have that the roots of the minimal polynomial of $A^n$ over $\Fp(F^\ell)$, say $r(x) \in \Fp(F^\ell)[x]$, are either a non-negative integer power of $F^\ell$ or multiplicatively independent with respect to $F$. Indeed, note that the roots of the minimal polynomial of $A$ are integral over $\Fp[F^{\ell}]$ (see also Section~\ref{subsec:H-C}) and so, if a root $u_0$ is multplicatively dependent with respect to $F^\ell$, then a power $u_0^n$ (for $n\in\N$) must be of the form $F^{\ell j_0}$ for some non-negative integer $j_0$.

So, with the above assumption regarding $A^n$ and its minimal polynomial $r(x)$, then we can write $r(x) = r_0(x)r_1(x)$ where $r_0(x)$ is a polynomial whose roots are (non-negative integer) powers of $F^\ell$ and $r_1(x)$ is a polynomial whose roots are multiplicatively independent with respect to $F$. Using Facts~\ref{fact:background}~and~\ref{fact:jordan-normal-form} (see Section~\ref{subsec:background}) along with Corollary \ref{cor:skewField}, there must exist an invertible matrix $P \in M_{N, N}(K)$ such that 
\begin{equation}
\label{eqn:splitA^n}
PA^nP^{-1} = A_0 \oplus A_1    
\end{equation} where $A_0$ corresponds to a matrix of the form \eqref{eqn:F-jordan} and the minimal polynomial of $A_1$ over  
$\Fp(F^\ell)$ is $r_1$. Using Equation \eqref{eqn:splitA^n} we have 
\begin{equation}
\label{eqn:splitA^mn}
PA^{mn}P^{-1} = A_0^m \oplus A_1^m,       
\end{equation}
for every positive integer $m$. Due to the definition of $K$ there exists a nonzero $u\in \Fp[F^\ell]$ such that $uP\in \Fq[F]$; also, because $A_1$ is integral over $\Fp[F^\ell]$, there exists a nonzero $h \in \Fp[F^\ell]$ such that each $hA_1^{m}$ (for $m\in\N$) has entries in $\Fq[F]$. Therefore, if we let $g$ be the group endomorphism corresponding to the matrix $uP$, using equation~\eqref{eqn:splitA^mn} we will get a commutative diagram of the form \eqref{eqn:splitPhi}. This concludes our proof of Proposition~\ref{prop:split}.
\end{proof}

The following result is an easy consequence of Proposition~\ref{prop:split} and of the fact that for any positive integer $a$, we have that $\binom{p^a}{i}=0$ in $\Fp$ whenever $0<i<p^a$.
\begin{cor}
\label{cor:split2}
In Proposition~\ref{prop:split}, at the expense of replacing the positive integer $n$ by a multiple, we may assume without loss of generality that $\Phi_0$ corresponds to a diagonal matrix of the form 
\[
A_0 =  F^{n_1\ell}\mathbf{I}_{m_1} \bigoplus \cdots \bigoplus  F^{n_s\ell}\mathbf{I}_{m_s}
\]
where, $n_1, \dots, n_s$ are distinct non-negative integers and $m_1, \dots, m_s$ are non-negative integers. 
\end{cor}
\begin{proof}
Let $p^a$ be a power of $p$ that is greater than all $m_1, \dots, m_s$ in the statement of Proposition \ref{prop:split}. Then, replacing $n$ by $np^a$ and combining the Jordan blocks corresponding to the same power of $F$ will deliver the desired conclusion. 
\end{proof}

Let $\Phi$ be a dominant endomorphism of $\bG_a^N$, let $n\in\N$, let $h \in \Fp[F^\ell]$ and  $\Phi_1:\bG_a^{N_1} \lra \bG_a^{N_1}$ be as in the statement of Proposition \ref{prop:split}, while $\Phi_0:\bG_a^{N_0} \lra \bG_a^{N_0}$ has the form as in Corollary~\ref{cor:split2}. With the above notation, we prove the next three technical lemmas. 

\begin{lemma}
\label{lem:condA}
Suppose that there exists some $i\in\{1,\dots, s\}$ such that $n_i=0$, and also $m_i>0$; in particular, this means that $N_0\ge 1$ with the notation as in Corollary~\ref{cor:split2}. Then there exists a non-constant rational function $f:\bG_a^N \lra \bP^1$ such that $f \circ \Phi^{n} = f$.
\end{lemma}
\begin{proof}
Suppose without loss of generality that $n_1 = 0$. Let $\pi: \bG_a^N \lra \bG_a$ be the projection onto the first coordinate. Then, using Equation \eqref{eqn:splitPhi} we must have
\[
\pi \circ g \circ [h] \circ \Phi^{n} = [h] \circ \pi \circ g.
\]
However, since $\pi$ and $g$ are both defined over $\Fq$, the map $[h]$ commutes with both of them. So, we have
\[
\pi \circ g \circ [h] \circ \Phi^{n} =  \pi \circ g \circ [h].
\]
Hence, $\pi \circ g \circ [h]$ defines a non-constant rational function that is left invariant by $\Phi^{n}$. 
\end{proof}

\begin{lemma}
\label{lem:condB}
Suppose that the numbers $n_1, \dots, n_s$ are all positive and $\max\{m_1, \dots, m_s\} \ge \trdeg_{\Fpbar} L + 1$. Then there exist integers $r\ge 1$ and $M\ge \trdeg_{\Fpbar}L+1$, and there exists a dominant group homomorphism $\tau: \bG_a^{N} \lra \bG_a^{M}$ such that $\tau \circ \Phi^n = F^r \circ \tau$. 
\end{lemma}
\begin{proof}
Suppose without loss of generality that $m_1 \ge \trdeg_{\Fpbar} L + 1$. Let $\pi$ be the projection map onto the first $m_1$ coordinates of $\bG_a^{N}$. Using the equation~\eqref{eqn:splitPhi} we must have
\[
\pi \circ g \circ [h] \circ \Phi^{n} = [h] \circ F^{n_1\ell} \circ \pi \circ g.
\]
Since  $g$, and $\pi$ are defined over $\Fq$, they must commute with $[h]$; also, they all commute with $F^{n_1\ell}$. Hence, we have
\[
\pi \circ g \circ [h] \circ \Phi^{n}  = F^{n_1\ell} \circ \pi \circ g \circ [h].
\]
So, the map $\tau := \pi \circ g \circ [h]$ has the desired property. 
\end{proof}

\begin{lemma}
\label{lem:condC}
Let $L$ be an algebraically closed field of characteristic $p$. 
If there exists a point $\alpha := (\alpha_0, \alpha_1)$ with $\alpha_0 \in \bG_a^{N_0}(L)$ and $\alpha_1 \in \bG_a^{N_1}(L)$ such that $$\OO := \left\{\left([h] \circ \left(\Phi_0^m, \Phi_1^m\right)\right)(\alpha_0, \alpha_1): m \ge 0\right\}$$
is Zariski dense in $\bG_a^{N_0 + N_1}$, then there also exists a point $\beta \in \bG_a^N(L)$ such that $\OO_{\Phi}(\beta)$ is Zariski dense in $\bG_a^N$.
\end{lemma}
\begin{proof}
Choose $\beta$ such that $g(\beta) = \alpha$ (note that $g$ is a dominant group endomorphism). Then the commutative diagram~\eqref{eqn:splitPhi} along with the fact that $g$ and $[h]$ are dominant group endomorphisms yields that the orbit of $\beta$ under $\Phi^n$ must be Zariski dense in $\bG_a^N$. Since $\OO_{\Phi^n}(\beta)\subseteq \OO_{\Phi}(\beta)$, we obtain the desired conclusion in Lemma~\ref{lem:condC}. 
\end{proof}

Lemmas \ref{lem:condA}, \ref{lem:condB}, and \ref{lem:condC} along with Proposition \ref{prop:iterate-reduction} will reduce Theorem \ref{thm:main} to Proposition \ref{prop:split2} stated and proved in the next Section.


\section{Proof of Theorem~\ref{thm:main}}
\label{sec:mainThm}

In this Section we conclude the proof of our main result. 
We work under the hypotheses of Theorem~\ref{thm:main}. We start by stating a useful result, which  is a special case of \cite[Proposition 4.1]{G-Sina-20}.

\begin{proposition}
\label{prop:F-diag}
Let $L$ be an algebraically closed field of  transcendence degree $d>0$ over $\Fp$. Let $\Phi: \bG_a^N \lra \bG_a^N$ be a dominant group endomorphism corresponding to the matrix
\[
A =  F^{n_1}\mathbf{I}_{m_1} \bigoplus \cdots \bigoplus  F^{n_s}\mathbf{I}_{m_s}
\]
where $m_1, \dots, m_s, n_1, \dots, n_s$ are positive integers and  $n_1, \dots, n_s$ are distinct. Then there exists a point  $\alpha\in\bG_a^N(L)$ such that every infinite subset of $\OO_{\Phi}(\alpha)$ is Zariski dense in $\bG_a^N$ if and only if 
\begin{equation}
    \max\{m_1, \dots, m_s\} \le d.
\end{equation}
\end{proposition}

Finally, we can state the technical reformulation of  Theorem~\ref{thm:main}, which will allow us to prove the desired conclusion in our main result. 
\begin{proposition}
\label{prop:split2} 
Let $N_0$ and $N_1$ be non-negative integers, let $q:=p^\ell$, let $L$ be an algebraically closed field which has transcendence degree over $\Fp$ equal to $d>0$, and let $K = \Fq[F] \otimes_{\Fp[F^\ell]} \Fp(F^\ell)$. Let   $\Phi_0: \bG_a^{N_0} \lra \bG_a^{N_0}$ be a dominant group endomorphism corresponding to the matrix
\begin{equation}
\label{eqn:F-diag}
A := F^{n_1}\mathbf{I}_{m_1} \bigoplus \cdots \bigoplus  F^{n_s}\mathbf{I}_{m_s},
\end{equation}
(for some non-negative integers $s,n_1,\dots, n_s$, while $N_0=\sum_{i=1}^s m_i$) 
and $\Phi_1: \bG_a^{N_1} \lra \bG_a^{N_1}$ be a finite-to-finite map  corresponding to a matrix $A_1 \in M_{N_1, N_1}(K)$,  where the minimal polynomial of $A_1$ over $\Fp[F^\ell]$ has roots that are multiplicatively independent with respect to $F^\ell$. Suppose there exists a non-zero element $h \in \Fp[F^\ell]$ such that $[h] \circ (\Phi_0^n, \Phi_1^n)$ is a well-defined dominant group endomorphism of $\bG_a^{N_0 + N_1}$ for each $n\in\N$. Then, one of the following statements must hold:
\begin{itemize}
    \item[($i$)] $N_0\ge 1$ and one of the numbers $n_1, \dots, n_s$ is equal to zero. 
    
    \item[($ii$)] The numbers $n_1, \dots, n_s$ are all positive and $\max\{m_1, \dots, m_s\} > d$.
    
    \item[$(iii)$] There exists a point $\alpha := (\alpha_0, \alpha_1)$ with $\alpha_0 \in \bG_a^{N_0}(L)$ and $\alpha_1 \in \bG_a^{N_1}(L)$ such that 
\begin{equation}
\label{eq:O}
\OO := \left\{\left([h] \circ \left(\Phi_0^n, \Phi_1^n\right)\right)(\alpha_0, \alpha_1): n \ge 0\right\}
\end{equation} 
is Zariski dense in $\bG_a^{N_0 + N_1}$.
\end{itemize}
\end{proposition}

As noted at the end of Section \ref{sec:reductions}, lemmas \ref{lem:condA}, \ref{lem:condB}, \ref{lem:condC} reduce Theorem~\ref{thm:main} to Proposition \ref{prop:split2}, which we will prove next.

\begin{proof}[Proof of Proposition~\ref{prop:split2}.]
First of all, as noted also in Section~\ref{sec:reductions}, for a point $\gamma\in \bG_a^k(L)$ (for some non-negative integer $k$), we will use the notation $\vec{\gamma}$ in order to emphasize that the point $\vec{\gamma}\in \bG_a^k(L)$ is a vector consisting of $k$ elements from $L$. 

We will prove Proposition~\ref{prop:split2} by assuming that if conditions $(i)$ and $(ii)$ do not hold, then condition $(iii)$ must hold. If we assume that conditions $(i)$ and $(ii)$ do not hold, then by Proposition \ref{prop:F-diag} there must exist a point $\vec{\alpha}_0 \in \bG_a^{N_0}(L)$ such that any infinite subset of  
\[
\OO_{\Phi_0}(\vec{\alpha}_0) := \{\Phi_0^n(\vec{\alpha}_0): n \ge 0\},
\]
is Zariski dense in $\bG_a^{N_0}$. Now, choose a point $\vec{\alpha}_1 \in \bG_a^{N_1}(L)$ whose coordinates are linearly independent with respect to the coordinates of $\vec{\alpha}_0$ over $\Fpbar[F]$ (see Proposition~\ref{prop:independent}). Note that if $N_0=0$, then our only requirement is that the coordinates of $\vec{\alpha}_1$ are linearly independent over $\Fpbar[F]$ (see the second part of Definition~\ref{def:independent}).

We let $\vec{\alpha} := \left(\vec{\alpha}_0, \vec{\alpha}_1\right)$. Suppose for the sake of contradiction that the Zariski closure of $\OO$ (from equation~\eqref{eq:O}) in $\bG_a^{N_0 + N_1}$ is a proper subvariety, say $V$.  Let \[
\Gamma := \left\{(B_0\vec{\alpha}_0, B_1 \vec{\alpha}_1): B_0 \in M_{N_0, N_0}(\Fq[F]) \text{ and } B_1 \in M_{N_1, N_1}(\Fq[F])\right\}.
\]
Then $\Gamma$ is a finitely generated $\Fp[F]$-module that contains $\OO$. Therefore, according to Proposition~\ref{prop:M-L}, the intersection $V(L) \cap \Gamma$ is contained in the union of finitely many sets of the form
\begin{equation}
\label{eqn:F-set}
\vec{\beta} + S\left(\vec{\gamma}_1, \dots, \vec{\gamma}_r; \delta_1, \dots, \delta_r\right) + H,    
\end{equation}
where $H$ is an $\Fp[F]$-submodule of $\Gamma$, and $S\left(\vec{\gamma}_1,\dots, \vec{\gamma}_r;\delta_1,\dots,\delta_r\right)$ is a sum of $F$-orbits of the points $\vec{\gamma}_i\in \bG_a^{N_0+N_1}(L)$ (for some given positive integers $\delta_i$, as in equation~\eqref{eq:S-additive}), i.e.,
$$S\left(\vec{\gamma}_1,\dots, \vec{\gamma}_r;\delta_1,\dots, \delta_r\right)=\left\{\sum_{i=1}^r F^{n_i\delta_i}\left(\vec{\gamma}_i\right)\colon n_i\in\N\text{ for }i=1,\dots, r\right\}.$$
Furthermore, as noted in Remark~\ref{rem:important_M-L} (see equation~\eqref{eq:in_Gamma}), there exists a nonzero polynomial $P(F)\in\Fp[F]$ such that 
\begin{equation}
\label{eq:22}
P(F)\left(\vec{\beta}\right) := \left(B_0\vec{\alpha}_0, B_1\vec{\alpha}_1\right),
\end{equation}
and for each $i=1,\dots, r$, we have 
\begin{equation}
\label{eq:33}
P(F)\left(\vec{\alpha}_i\right) := \left(C_{0,i}\vec{\alpha}_0, C_{1, i}\vec{\alpha}_1\right),
\end{equation}
for some $B_0, C_{0,1}, \dots, C_{0, r} \in M_{N_0, N_0}(\Fq[F])$ and $B_1, C_{1,1}, \dots, C_{1, r} \in M_{N_1, N_1}(\Fq[F])$. Furthermore, since $\Fp[F]$ is a finite integral extension of $\Fp[F^\ell]$, then (at the expense of multiplying $P(F)$ by a suitable nonzero element of $\Fp[F]$, which would only replace the matrices $B_i$ and $C_{i,j}$ by other matrices with entries in $\Fq[F]$) we may assume that $P(F)\in\Fp[F^\ell]$. 

We let $U$ be a set of the form \eqref{eqn:F-set} that contains the subset 
\begin{equation}
\label{eq:elements_equation}
\OO_S := \left\{\left([h] \circ \left(\Phi_0^n, \Phi_1^n\right)\right)(\alpha_0, \alpha_1): n \in S\right\},
\end{equation}
of $\OO$ where $S$ is a subset of $\N$ that has a positive natural density. 

Now, since $H$ is an $\Fp[F]$-submodule of $\bG_a^{N_0+N_1}$, then its Zariski closure $\overline{H}$ is an algebraic subgroup of $\bG_a^{N_0+N_1}$ defined over $\Fp$. So,  
let $\vec{v} = (\vec{v}_0, \vec{v}_1)$ with $\vec{v}_0 \in \Fp[F]^{N_0}$ and $\vec{v}_1 \in \Fp[F]^{N_1}$ such that 
\begin{equation}
\label{eq:eq-H}
\vec{v}_0^T \vec{x}_0 + \vec{v}_1^T \vec{x}_1 = 0,
\end{equation}
for all $(\vec{x}_0, \vec{x}_1) \in \overline{H}$ (where always $\vec{v}^T$ denotes the transpose of $\vec{v}$). Note that since $\overline{H}$ is an algebraic subgroup of $\bG_a^{N_0+N_1}$ defined over $\Fp$, then $\overline{H}$ is the zero locus of finitely many equations of the form~\eqref{eq:eq-H}. Using both equations~\eqref{eq:eq-H}~and~\eqref{eqn:F-set} along with equations~\eqref{eq:22}~and~\eqref{eq:33} and with the fact that the operator $P(F)$ leaves invariant the entries in both $\vec{v}_0$ and $\vec{v}_1$, we obtain  that 
\begin{equation}
 \sum_{i=0}^1P(F)\left(\vec{v}_i^T \cdot \left([h]\circ\Phi_i^n\right)(\vec{\alpha}_i)\right) = \sum_{i=0}^1\left(\vec{v}_i^TB_i\vec{\alpha}_i + \sum_{j=1}^r \vec{v}_i^TF^{n_j\delta_j}C_{i,j}\vec{\alpha}_j\right). 
\end{equation}
for all $n \in S$. Since for each $j=0,1$, we have that $\Phi_j$ corresponds to the matrix $A_j\in M_{N_j,N_j}(K)$, and then using that the set of the coordinates of $\vec{\alpha}_1$ are linearly independent from the set of  coordinates of $\vec{\alpha}_0$ over $\Fpbar[F]$, then writing $h_1:=P(F)\cdot h\in \Fp[F^\ell]$, we get:
\begin{equation}
\label{eqn:sum-of-F-powers-1}
\vec{v}_1^T \left(h_1A_1^n - B_1 - \sum_{j=1}^r F^{n_j\delta_j}C_{1,j}\right) = \vec{0}\in M_{N_1,1}(K)\text{ for all }n\in S.
\end{equation}

At the expense of replacing $S$ with a subset of $S$ with a positive natural density, we may assume that each $n_j\delta_j$ (for $j=1,\dots, r$) has the same remainder modulo $\ell$ for all $n \in S$. This  allows us to rewrite equation~\eqref{eqn:sum-of-F-powers-1} as an equation of the form
\begin{equation}
\vec{v}_1^T  \left(h_1A_1^n - B_1 - \sum_{j=1}^r F^{m_j\ell}C'_{1,j}\right) = \vec{0},    
\end{equation}
for some matrices $C'_{1,j}\in M_{N_1,N_1}(\Fq[F])$ depending on the matrices $C_{1,j}$.  
Let $\mathcal{V}$ be the $N_1\times N_1$-matrix whose rows are all equal to $\vec{v}_1^T$. So, we have
\begin{equation}
\label{eq:this}
\mathcal{V}  \left(h_1A_1^n - B_1 - \sum_{j=1}^r F^{m_j\ell}C'_{1,j}\right) = 0\in M_{N_1,N_1}(K).    
\end{equation}
Applying the operator $\sim$ (defined as in Notation~\ref{not:not} from Section~\ref{subsec:matrices}) to equation~\eqref{eq:this} and also using the fact that $h_1\in\Fp[F^\ell]$, we get
\begin{equation}
\label{eqn:tildeEqn}
\tilde{\mathcal{V}}  \left(h_1\tilde{A}_1^n - \tilde{B}_1 - \sum_{j=1}^r F^{m_j\ell}\tilde{C'}_{1,j}\right) = 0.    
\end{equation}

Now suppose that $\mathcal{V}$ is non-zero. Then we must have
some nonzero row $\vec{u}^T$  in $\tilde{\mathcal{V}}$; so, we get 
\begin{equation}
\label{eqn:tildeEqn1}
\vec{u}^T \left(h_1\tilde{A}_1^n - \tilde{B}_1 - \sum_{j=1}^r F^{m_j\ell}\tilde{C'}_{1,j}\right) = 0.    
\end{equation}
We note that equation~\eqref{eqn:tildeEqn1} is similar to  \cite[Lemma~4.5,~equation~(4.5.1)]{G-Sina-20} (after transposing both sides). So, by proceeding exactly as in the proof of \cite[Lemma 4.5]{G-Sina-20}, and using the fact that the roots of the minimal polynomial of $A_1$ (and thus, also of $\tilde{A}_1$) are multiplicatively independent with respect to $F^\ell$, equation~\eqref{eqn:tildeEqn1} leads to an equation of the form
\begin{equation}
\label{eqn:lambda-F}
a\lambda^n =c_0 +\sum_{j=1}^rc_jF^{m_j\ell},    
\end{equation}
for some $a, c_0, c_1, \dots, c_r \in \overline{\Fp(F^\ell)}$ with $a \ne 0$ and some eigenvalue $\lambda$ of $A_1$, which is thus multiplicatively independent with respect to $F^\ell$. Note that $\overline{\Fp(F^\ell)}$ is isomorphic to $\overline{\Fp(t)}$ for some transcendental element $t$ since $\Fp(F^\ell)$ is naturally isomorphic to $\Fp(t)$. However, by Lemma~\ref{lem:lambda^m}, the set of $n$'s for which equation~\eqref{eqn:lambda-F} is solvable for some positive integers $m_j$ must have a natural density equal to zero which contradicts our choice of the subset $S$. 

This means that for any vector $\vec{v} = (\vec{v}_0, \vec{v}_1)$ with $\vec{v}_0 \in \Fp[F]^{N_0}$ and $\vec{v}_1 \in \Fp[F]^{N_1}$ such that $
\vec{v}_0^T \vec{x}_0 + \vec{v}_1^T \vec{x}_1 = 0,
$ for all $(\vec{x}_0, \vec{x}_1) \in \overline{H}$ we must have $\vec{v}_1 = \vec{0}$. Hence, $\overline{H} := G_0 \times \bG_a^{N_1}$ for some algebraic subgroup $G_0 \subseteq \bG_a^{N_0}$. 

Thus, $\overline{U}$ (the Zariski closure of the set $U$ containing the elements from equation~\eqref{eq:elements_equation}), which is itself a subset of $V$,  must be a set of the form $W \times \bG_a^{N_1}$ for some closed subset $W \subseteq \bG_a^{N_0}$ since $\{0\} \times \bG_a^{N_1}$ is contained in the stabilizer of $\overline{U}$ (because $\overline{H}=G_0\times \bG_a^{N_1}$). On the other hand, note that $W$ contains 
\[
\left\{([h] \circ \Phi_0^n)\left(\vec{\alpha}_0\right): n \in S \right\},
\]
which must be Zariski dense in $\bG_a^{N_0}$ because of our choice of $\vec{\alpha}_0$ and the fact that $[h]$ is a dominant endomorphism of $\bG_a^{N_0}$. So, we conclude that $\overline{U} = \bG_a^{N_0} \times \bG_a^{N_1}$ which contradicts the fact that $V$ is a proper subvariety of $\bG_a^{N_0} \times \bG_a^{N_1}$. This contradiction completes our proof for Proposition \ref{prop:split2}. 
\end{proof}

Since we proved that Theorem~\ref{thm:main} reduces to  Proposition~\ref{prop:split2}, this concludes our proof for Theorem~\ref{thm:main}.


\end{document}